\documentclass[a4paper,11pt,reqno]{amsart}

\usepackage[latin1]{inputenc}

\usepackage{
amssymb,
tikz,
xcolor
}

\usepackage[hidelinks]{hyperref}

\usepackage{geometry}
\newtheorem{thm}{Theorem}[section]

\newtheorem{prop}[thm]{Proposition}

\theoremstyle{remark}
\newtheorem{remark}[thm]{Remark}

\theoremstyle{definition}
\newtheorem{defi}[thm]{Definition}

\newcommand{\R}{\mathbb{R}}
\newcommand{\C}{\mathbb{C}}
\newcommand{\Z}{\mathbb{Z}}

\newcommand{\mE}{\mathrm{E}}
\newcommand{\mV}{\mathrm{V}}

\newcommand{\cK}{\mathcal{K}}

\newcommand{\cD}{\mathcal{D}}

\newcommand{\cG}{\mathcal{G}}

\renewcommand\phi{\varphi}

\newcommand{\vv}{\mathrm{v}}

\newcommand{\dom}{\mathrm{dom}}
\renewcommand{\geq}{\geqslant}
\renewcommand{\leq}{\leqslant}
\newcommand{\Sum}{\displaystyle \sum}
\newcommand{\f}[2]{\frac{#1}{#2}}

\renewcommand{\geq}{\geqslant}
\renewcommand{\leq}{\leqslant}

\newcommand\gH{{\mathfrak{H}}}
\newcommand{\gG}{{\Gamma}}

\newcommand\cH{{\mathcal{H}}}

\newcommand\cA{{\mathcal{A}}}

\newcommand\cN{{\mathfrak{N}}}

\newcommand\I{{\rm{i}}}

\newcommand{\be}{\begin{equation}}
\newcommand{\ee}{\end{equation}}

\tikzstyle{nodino}=[circle,draw,fill,inner sep=0pt,minimum size=0.5mm]
\tikzstyle{infinito}=[circle,inner sep=0pt,minimum size=0mm]
\tikzstyle{nodo}=[circle,draw,fill,inner sep=0pt, minimum size=0.5*width("k")]
\tikzstyle{nodo_vuoto}=[circle,draw,inner sep=0pt, minimum size=0.5*width("k")]
\usetikzlibrary{graphs}
\tikzset{every loop/.style={min distance=10mm,in=300,out=240,looseness=10}}
\tikzset{place/.style={circle,thick,draw=blue!75,fill=blue!20,minimum
size=6mm}}
\tikzset{place2/.style={circle,thick,draw=red!75,fill=red!20,minimum
size=6mm}}

\title[The Dirac-Kirchoff operator on metric graphs]{A note on the Dirac operator with Kirchoff-type\\ vertex conditions on metric graphs}

\author[W. Borrelli]{William Borrelli}
\address[W. Borrelli]{Scuola Normale Superiore, Centro De Giorgi, Piazza dei Cavalieri 3, I-56100 , Pisa, Italy.} 
\email{william.borrelli@sns.it}

\author[R. Carlone]{Raffaele Carlone}
\address[R. Carlone]{Universit\`{a} ``Federico II'' di Napoli, Dipartimento di Matematica e Applicazioni ``R. Caccioppoli'', MSA, via Cinthia, I-80126, Napoli, Italy.} 
\email{raffaele.carlone@unina.it}

\author[L. Tentarelli]{Lorenzo Tentarelli}
\address[L. Tentarelli]{Politecnico di Torino, Dipartimento di Scienze Matematiche ``G.L. Lagrange'', Corso Duca degli Abruzzi 24, 10129, Torino, Italy.} 
\email{lorenzo.tentarelli@polito.it}

\begin{document}

%
%

\begin{abstract}
In this note we present some properties of the Dirac operator on noncompact metric graphs with \emph{Kirchoff-type} vertex conditions. In particular, we discuss its spectral features and describe the associated quadratic form. Moreover, we prove that the operator converges to the Schr\"odinger operator with usual Kirchoff conditions, in the non-relativistic limit.
\end{abstract}
\maketitle

\label{sec-intro}

The investigation of evolution equations on metric graphs (see, e.g., Figure \ref{fig-gen}) has become very popular nowadays as they are assumed to represent effective models for the study of the dynamics of physical systems confined in branched spatial domains. A considerable attention has been devoted to the cubic Schr\"odinger equation, as it is supposed to well approximate the behavior of Bose-Einstein condensates in ramified traps (see, e.g., \cite{GW-PRE} and the references therein).

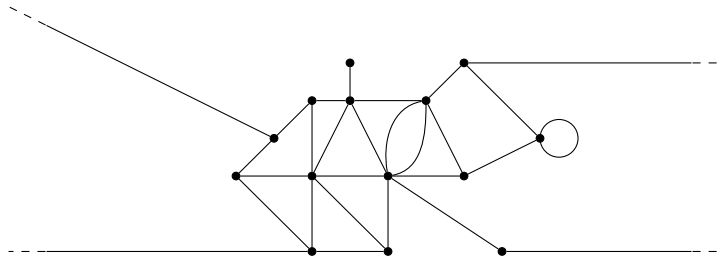
\begin{figure}
\centering
\begin{tikzpicture}[xscale= 0.5,yscale=0.5]
\node at (0,0) [nodo] (1) {};
\node at (2,0) [nodo] (2) {};
\node at (4,1) [nodo] (3) {};
\node at (2,3) [nodo] (4) {};
\node at (1,2) [nodo] (5) {};
\node at (-1,2) [nodo] (6) {};
\node at (-2,0) [nodo] (7) {};
\node at (-1,3) [nodo] (8) {};
\node at (-2,2) [nodo] (9) {};
\node at (-4,0) [nodo] (10) {};
\node at (-2,-2) [nodo] (11) {};
\node at (0,-2) [nodo] (12) {};
\node at (3,-2) [nodo] (13) {};
\node at (-3,1) [nodo] (14) {};
\node at (8,-2) [infinito] (15) {};
\node at (9,-2) [infinito] (15b) {};
\node at (8,3) [infinito] (16) {};
\node at (9,3) [infinito] (16b) {};
\node at (-9,-2) [infinito] (17) {};
\node at (-10,-2) [infinito] (17b) {};
\node at (-9,4) [infinito] (18) {};
\node at (-10,4.5) [infinito] (18b) {};
\draw (4.5,1) circle (0.5cm);
\draw [-] (1) -- (2);
\draw [-] (2) -- (3);
\draw [-] (3) -- (4);
\draw [-] (4) -- (5);
\draw [-] (5) -- (2);
\draw [-] (1) -- (13);
\draw [-] (1) -- (12);
\draw [-] (5) -- (6);
\draw [-] (1) -- (6);
\draw [-] (8) -- (6);
\draw [-] (11) -- (12);
\draw [-] (7) -- (12);
\draw [-] (1) -- (7);
\draw [-] (6) -- (7);
\draw [-] (9) -- (7);
\draw [-] (7) -- (11);
\draw [-] (11) -- (10);
\draw [-] (7) -- (10);
\draw [-] (9) -- (10);
\draw [-] (9) -- (6);
\draw [-] (1) to [out=100,in=190] (5);
\draw [-] (1) to [out=10,in=270] (5);
\draw [-] (13) -- (15);
\draw [dashed] (15) -- (15b);
\draw [-] (4) -- (16);
\draw [dashed] (16) -- (16b);
\draw [-] (14) -- (18);
\draw [dashed] (18) -- (18b);
\draw [-] (11) -- (17);
\draw [dashed] (17) -- (17b);
\end{tikzpicture}
\caption{A general noncompact metric graph.}
\label{fig-gen}
\end{figure}


This, naturally, has lead to the study of the graph versions of the laplacian, given by suitable \emph{vertex conditions} and, especially, to the study of the standing waves of the associated NonLinear Schr\"odinger Equation (NLSE) (see, e.g., \cite{ADST-APDE,AST-CVPDE,AST-JFA,AST-CMP,AST-CVPDE2,CDS-MJM,CFN-Non,D-JDE,DD-p,DT-OTAA,DT-p,ST-JDE,ST-NA,T-JMAA}). In particular, the most investigated sub case has been that of the \emph{Kirchhoff} vertex conditions, which impose at each vertex $\vv$:
\begin{itemize}
 \item[(i)] continuity of the function: $u_e(\vv)=u_f(\vv),\qquad\forall e,f\succ\vv,\qquad\forall \vv\in\cK$,\\[-.3cm]
 \item[(ii)] ``balance'' of the derivatives: $\sum_{e\succ v}\f{du_e}{dx_e}(\vv)=0,\qquad\forall \vv\in\cK$,
\end{itemize}
where $\cK$ denotes the \emph{compact core} of the graph (i.e., the subgraph of the bounded edges), $e\succ\vv$ means that the edge $e$ is incident at the vertex $\vv$ and $\f{du_e}{dx_e}(\vv)$ stands for $u_e'(0)$ or $-u_e'(-\ell_e)$ according to the parametrization of the edge (for more see Section \ref{sec-set}). The above conditions correspond to the \emph{free case}, namely, where there is no interaction at the vertices which are then mere junctions between edges. 

As a further development, in the last years the study of the Dirac operator on metric graphs has also generated a growing interest (see, e.g., \cite{ALTW-IEOT,BH-JPA,BT-JMP,P}). Moreover, recently \cite{SBMK-JPA} proposed (although if in the prototypical case of the \emph{infinite 3-star} graph depicted in Figure \ref{fig-3star}) the study of the NonLinear Dirac Equation (NLDE) on networks, where  the Dirac operator is given by
\begin{equation}
 \label{eq-dirac_formal}
 \cD:=-\imath c\frac{d}{dx}\otimes\sigma_{1}+mc^{2}\otimes\sigma_{3}\,.
\end{equation}
Here $m>0$ and $c>0$ are two parameters representing the \emph{mass} of the generic particle of the system and the \emph{speed of light} (respectively), and $\sigma_1$ and $\sigma_3$ are the so-called \emph{Pauli matrices}, i.e.
\begin{equation}
 \label{eq-pauli}
 \sigma_1:=\begin{pmatrix}
  0 & 1 \\
  1 & 0
 \end{pmatrix}
 \qquad\text{and}\qquad
 \sigma_3:=\begin{pmatrix}
  1 & 0 \\
  0 & -1
 \end{pmatrix}.
\end{equation}
Precisely, as for the majority of the works on the NLSE, \cite{SBMK-JPA} suggests the study of the stationary solutions, that is those $2$-spinors such that $\chi(t,x)=e^{-i\omega t}\,\psi(x)$, with $\omega\in\R$ and with $\psi$ solving
\[
 \cD\psi-|\psi|^{p-2}\,\psi=\omega\, \psi.
\]

\begin{remark}
We observe that, strictly speaking, the parameter $c$ in \eqref{eq-dirac_formal} corresponds to the \emph{speed of light} only in truly relativistic models, whereas in other contexts it should be rather considered as a phenomenological parameter depending on the model under study. Nevertheless, for the sake of simplicity, we will refer to it as ``speed of light'' throughout the paper.
\end{remark}


\begin{figure}
 \centering
 \begin{tikzpicture}[xscale= 0.3,yscale=0.3]
 \node at (0,0) [nodo] (00) {};
 \node at (6,0) [infinito] (60) {};
 \node at (8,0) [infinito] (70) {};
 \node at (-4,4) [infinito] (-44) {};
 \node at (-6,6) [infinito] (-4545) {};
 \node at (-4,-4) [infinito] (-4-4) {};
 \node at (-6,-6) [infinito] (-45-45) {};

 \draw [-] (00) -- (60);
 \draw [dashed] (70) -- (60);
 \draw [-] (00) -- (-44);
 \draw [dashed] (-44) -- (-4545);
 \draw [-] (00) -- (-4-4);
 \draw [dashed] (-4-4) -- (-45-45);
 \end{tikzpicture}
 \caption{The infinite 3-star graph}
 \label{fig-3star}
\end{figure}
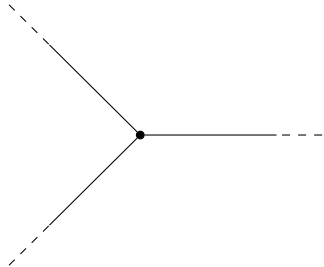

The attention recently attracted by the linear and the nonlinear Dirac equations is due to their physical applications, as effective equations, both in solid state physics and in nonlinear optics (see \cite{HC-PD,HC-NJP}). While initially the NLDE appeared as a field equation for relativistic interacting fermions (see \cite{ELS-BAMS,LKG-PRD}), thereafter it was used in particle physics (to simulate features of quark confinement), in acoustics and in the context of Bose-Einstein Condensates (see \cite{HC-NJP}). Recently, it also appeared that certain properties of some physical models, as thin carbon structures, are effectively described by the NLDE (see \cite{AS-JMP,B-JDE,B-JMP,B-CVPDE,CCNP-SIMA,FW-JAMS,FW-CMP, P-RIMS}).

On the other hand, in the context of metric graphs NLDE may describe the constrained dynamics of genuine relativist particles, or be regarded as an effective model for solid state/nonlinear optics systems (as already remarked). In particular, it applies in the analysis of effective models of condensed matter physics and field theory (see \cite{SBMK-JPA}). Moreover, Dirac solitons in networks may be realized in optics, in atomic physics, etc. (see again \cite{SBMK-JPA} and the references therein).

Concerning the existence of standing waves for the NLDE on metric graphs, to our knowledge, the first rigorous mathematical work on the subject is \cite{BCT-p}, where a nonlinearity localized on the compact core of the graph is considered (refer to \cite{BCT-pp} for a survey on the standing waves of NLDE and NLSE with localized nonlinearity). Recently, in \cite{BCT-non} the authors also studied the Cauchy problem for nonlinear Dirac equations with extended nonlinearities (for general metric graphs) and the existence of standing waves on infinite $N$-star graphs. 

\medskip
It is clear that (as for the NLSE), preliminarily to the study of the nonlinear case, it is necessary to find suitable vertex conditions for the operator $\cD$ that make it self-adjoint. In this paper, we consider those conditions that converge to the Kirchhoff ones in the nonrelativistic limit (for details see \cite[Appendix A]{BCT-non}), and that we call \emph{Kirchhoff-type}, which represents (as well as Kirchhoff for Schr\"odinger) the \emph{free case}. Roughly speaking, these conditions ``split'' the requirements of Kirchhoff conditions: the continuity condition is imposed only on the first component of the spinor, while the second component (in place of the derivative) has to satisfy the ``balance'' condition (see \eqref{eq-kirchtype1}$\&$\eqref{eq-kirchtype2} below).

Precisely, here we discuss self-adjointness of this graph realization of $\cD$, its more relevant spectral properties and some question arising in the rigorous definition of the associated quadratic form. Although this could seem a slightly detached topic with respect to zero-range interactions, it actually is one of its natural evolutions. The attachment of two (or more) edges of a metric graph is usually modeled, indeed, by means of an interplay between a point interaction on each of the two (or more) edges. As a consequence, the techniques to investigate self-adjoint extensions of differential operators on graphs relies on the very same techniques developed for the study self-adjoint extensions of differential operators with point/contact interactions. The Kirchhoff-type case, even representing the free case in the graph context, can be managed simply as a particular case in this framework.

\medskip
We limit ourselves to the case of \emph{noncompact} graphs with a finite number of edges since it is the most studied one in the nonlinear context.

\medskip
The paper is organized as follows. In Section \ref{sec-set} we recall some fundamental notions on metric graphs and we give the definition of the Dirac operator with \emph{Kirchhoff-type} vertex conditions. In Section \ref{sec-linear}, we give a justification of the self adjointness of the operator and of its spectral properties. A model case is presented in Section \ref{sec-model}, in order to clarify the general arguments developed in Section \ref{sec-linear}. Finally, Section \ref{sec-quadraticform} accounts for the particular features of the quadratic form associated with $\cD$ and its form domain. 


\section{Functional setting}
\label{sec-set}

Preliminarily, we recall some basic notions on metric graphs. More details can be found in \cite{AST-CVPDE,BK,K-WRM} and  references therein.

Throughout, by \emph{metric graph} $\cG=(\mV,\mE)$ we mean a connected {\em multigraph} (i.e., possibly with multiple edges and self-loops) with a finite number of edges and vertices. Each edge is a finite or half-infinite segment of real line and the edges are glued together at their endpoints (the vertices of $\cG$) according to the topology of the graph (see, e.g., Figure \ref{fig-gen}). 

Unbounded edges are identified with copies of $\R^+ = [0,+\infty)$ and are called half-lines, while bounded edges are identified with closed and bounded intervals $I_e =[0,\ell_e]$, $\ell_e>0$. Each edge, bounded or unbounded, is endowed with a coordinate $x_e$ which possess an arbitrary orientation when the interval is bounded and the natural orientation in case of a half-line. As a consequence, the graph $\cG$ is a locally compact metric space, the metric given by the shortest distance along the edges. Clearly, since we assume a finite number of edges and vertices, $\cG$ is \emph{compact} if and only if it does not contain any half-line. 

\medskip
Consistently, a function $u:\cG\to\C$ is actually a family of functions $(u_e)$, where $u_e:I_e\to\C$  is the restriction of $u$ to the edge $e$. The usual $L^p$ spaces can be defined in the natural way, with norms
\[
 \|u\|_{L^p(\cG)}^p := \sum_{e\in\mE} \| u_e\|_{L^p(I_e)}^p,\quad \text{if }p\in[1,\infty),
\]
and
\[
 \|u\|_{L^\infty(\cG)} := \max_{e\in\mE} \| u_e\|_{L^\infty(I_e)},
\]
while $H^m(\cG)$ are the spaces of functions $u=(u_e)$ such that $u_e\in  H^m(I_e)$ for every edge $e\in\mE$, with norm
\[
 \|u\|_{H^m(\cG)}^2 = \sum_{i=0}^m\|u^{(i)}\|_{L^2(\cG)}^2.
\]
Accordingly, a spinor $\psi=(\psi^1,\psi^2)^T:\cG\to\C^2$ is a family of 2-spinors
\[
 \psi_e=\begin{pmatrix}\psi_e^1\\[.2cm] \psi_e^2\end{pmatrix}:I_e\longrightarrow\C^{2},\qquad \forall e\in\mE,
\]
and thus
\[
L^{p}(\cG,\C^{2}):= \bigoplus_{e\in\mE} L^{p}(I_e,\C^{2}),
\]
endowed with the norms
\[
 \Vert\psi\Vert_{L^{p}(\cG,\C^{2})}^p:=\Sum_{e\in\mE}\Vert \psi_e\Vert_{L^{p}(I_e,\C^{2})}^p,\quad \text{if }p\in[1,\infty),
\]
and
\[
 \Vert\psi\Vert_{L^{\infty}(\cG,\C^2)}:=\max_{e\in\mE}\Vert\psi_e\Vert_{L^{\infty}(I_e,\C^{2})},
\]
whereas
\[
H^{m}(\cG,\C^{2}):= \bigoplus_{e\in\mE} H^{m}(I_e,\C^{2})
\]
endowed with the norm
\[
 \Vert\psi\Vert_{H^{m}(\cG,\C^{2})}^2:=\Sum_{e\in\mE}\Vert \psi_e\Vert_{H^{m}(I_e,\C^{2})}^2
\]

\begin{remark}
 Usually, graph Sobolev spaces are not defined as before. They also contain some further requirement on the behavior of the functions at the vertices of the graph (in the case $m=1$, for instance, global continuity is often required). However, in a Dirac context it is convenient to keep integrability requirements and conditions at the vertices separated.
\end{remark}
 
Now, we can define the Kirchhoff-type realization of the Dirac operator on graphs.

\begin{defi}
\label{defi-dirac}
Let $\cG$ be a metric graph and let $m,c>0$. We define the \emph{Dirac operator} with Kirchhoff-type vertex conditions the operator $\cD:L^2(\cG,\C^2)\to L^2(\cG,\C^2)$ with action
\begin{equation}
\label{eq-actionD}
\cD_{|_{I_e}}\psi=\cD_e\psi_e:=-\imath c\,\sigma_{1}\psi_e'+mc^{2}\,\sigma_{3}\psi_e,\qquad\forall e\in\mE,
\end{equation}
$\sigma_1,\sigma_3$ being the matrices defined in \eqref{eq-pauli}, and domain
\[
 \dom(\cD):=\left\{\psi\in H^1(\cG,\C^2):\psi\text{ satisfies \eqref{eq-kirchtype1} and \eqref{eq-kirchtype2}}\right\},
\]
where
\begin{gather}
 \label{eq-kirchtype1}
 \psi_e^{1}(\vv)=\psi^{1}_{f}(\vv),\qquad\forall e,f\succ\vv,\qquad\forall \vv\in\cK,\\[.4cm]
 \label{eq-kirchtype2}
 \sum_{e\succ v}\psi^{2}_{e}(\vv)_{\pm}=0,\qquad\forall \vv\in\cK,
\end{gather}
$\psi^{2}_{e}(\vv)_{\pm}$ standing for $\psi^{2}_{e}(0)$ or $-\psi^{2}_{e}(\ell_e)$ according to whether $x_e$ is equal to $0$ or $\ell_e$ at vertex $\mathrm{v}$.
\end{defi}

An immediate, albeit informal, way to see why the previous one can be considered a Kirchhoff-type realization of the Dirac operator is the following. First, recall that the domain of the Kirchhoff laplacian consists of the $H^2(\cG)$-functions that also satisfy conditions (i)$\&$(ii) of Section \ref{sec-intro}. Therefore, since, roughly speaking, the laplacian is (a component of) the square of the Dirac operator (up to a zero order term), one can square $\cD$ and check if the resulting operator is in fact the Kirchhoff laplacian (up to a zero order term). This is, indeed, the case, since $\cD^2$ clearly acts as $(-\Delta+m^2c^4)\otimes\mathbb{I}_{\C^2}$ and since, considering spinors of the type $\psi=(\psi^1,0)^T$, if one imposes that $\psi\in\dom(\cD^2)$ (namely, that $\psi\in\dom(\cD)$ and that $\cD\psi\in\dom(\cD)$), then $\psi^1\in\dom(-\Delta)$. 


\section{Self-adjointness and spectrum}
\label{sec-linear}

In this section, we prove the self-adjointness of the operator $\cD$ and present its main spectral features. 

Preliminarily, we observe that the proof of the self-adjointness is not new (see, e.g. \cite{AP-JPA,BT-JMP,N-RSTA,P-OM}). In particular, \cite{BT-JMP} shows it for a wide class of vertex conditions, including the Kirchoff-type ones. Here we give an alternative justification using the theory of \emph{boundary triplets}.

The study the of $\cD$ requires some introductory notions. Let $A$ be a densely defined closed symmetric operator in a separable Hilbert space $\mathcal{H}$ with equal deficiency
indices $\mathrm{n}_\pm(A):=\dim \cN_{\pm \imath} \leq \infty,$ where
\[
\cN_z:=\ker(A^*-z)\,,\qquad z\in\C\setminus\R\,,
\]
is the defect subspace.

\begin{defi}
A triplet $\Pi=\{\gH,\gG_0,\gG_1\}$ is said a \emph{boundary triplet} for the adjoint operator $A^*$ if and only if $\gH$ is a Hilbert space and $\Gamma_0,\Gamma_1:\  \dom(A^*)\rightarrow \gH$
are linear mappings such that
\[
\langle A^*f|g\rangle - \langle f|A^*g\rangle = \langle\gG_1f,\gG_0g\rangle_\gH-
\langle\gG_0f,\gG_1g\rangle_\gH, \qquad f,g\in\dom(A^*),
\]
holds (with $\langle\cdot,\cdot\rangle_\gH$ the scalar product in $\gH$) and the mapping
\[
 \gG:=\begin{pmatrix}\Gamma_0\\[.2cm]\Gamma_1\end{pmatrix}:  \dom(A^*)
\rightarrow \gH \oplus \gH
\]
is surjective.
\end{defi}

\begin{defi}
 Let $\Pi=\{\gH,\gG_0,\gG_1\}$ be a boundary triplet for the adjoint operator $A^*$, define the operator
 \[
  A_{0}:=A^{*}|_{\operatorname{ker}\Gamma_{0}}
 \]
and denote by $\rho(A_{0})$ its resolvent set. Then, the {\em $\gamma$-field} and {\em Weyl function} associated with $\Pi$ are the operator valued functions $\gamma(\cdot) :\rho(A_0)\rightarrow  \mathcal{L}(\gH,\cH)$ and
$M(\cdot):\rho(A_0)\rightarrow  \mathcal{L}(\gH)$, respectively, defined by
\begin{equation}
\label{gammafield}
\gamma(z):=\left(\Gamma_0|_{\cN_z}\right)^{-1}
\qquad\text{and}\qquad M(z):=\Gamma_1\circ\gamma(z), \qquad
z\in\rho(A_0).
\end{equation}
\end{defi}


\subsection{Self-adjointness}

In order to apply the theory of boundary triplets to the definition of the Kirchhoff-type Dirac operator, one has to study the operator on the single components of the graph (segments and half lines) imposing suitable boundary conditions. Then, one describes the effect of connecting these one-dimensional components, according to the topology of the graph, through the vertex conditions \eqref{eq-kirchtype1}-\eqref{eq-kirchtype2}.

\medskip
First, observe that the set $\mE$ of the edges of a metric graph $\cG$ can be decomposed in two subsets, namely, the set of the bounded edges $\mathrm{E}_{s}$ and set of the half-lines $\mathrm{E}_{h}$. Fix, then, $e\in \mathrm{E}_{s}$ and consider the corresponding minimal operator  $\widetilde{D_{e}}$ on $\mathcal{H}_{e}=L^{2}(I_{e},\C^{2})$, which has the same action of \eqref{eq-actionD} but domain $H^{1}_{0}(I_{e},\C^{2})$. As a consequence, the adjoint operator possesses the same action and domain
\[
 \dom(\widetilde{D^{*}_{e}})=H^{1}(I_{e},\C^{2}).
\]
On the other hand, a suitable choice of trace operators (introduced in \cite{GT-p}) is given by $\Gamma^{e}_{0,1}:H^{1}(I_{e},\C^{2})\rightarrow\C^{2}$, with
\[
\Gamma^{e}_{0}\begin{pmatrix}\psi^{1}_e \\[.2cm] \psi^{2}_e  \end{pmatrix}=\begin{pmatrix} \psi^{1}_e(0) \\[.2cm] \imath c\psi^{2}_e(\ell_{e})\end{pmatrix},\qquad \Gamma^{e}_{1}\begin{pmatrix}\psi^{1}_e \\[.2cm] \psi^{2}_e  \end{pmatrix}=\begin{pmatrix} \imath c\psi^{2}_e(0) \\[.2cm] \psi^{1}_e(\ell_{e})\end{pmatrix}
\]
and hence, given the boundary triplet $\left\{\gH_{e},\Gamma^{e}_{0},\Gamma^{e}_{1} \right\}$, with $\gH_{e}=\C^{2}$, one can compute the gamma field and the Weyl function using \eqref{gammafield} and prove that $\widetilde{D^{*}_{e}}$ has defect indices $n_{\pm}(\widetilde{\cD_{e}})=2$. Moreover, note that in this way we can define an operator $\cD_e$ with the same action of $\widetilde{D^{*}_{e}}$ (and $\widetilde{D_{e}}$) and domain
\[
\dom(\cD_{e})=\operatorname{ker}\Gamma^{e}_{0},
\]
which is self-adjoint by construction.

Analogously, fix now $e'\in \mathrm{E}_{h}$ and consider the minimal operator $\widetilde{\cD_{e'}}$ defined on $\cH_{e'}=L^{2}(\R_{+},\C^{2})$, with the same action as before and domain $H^{1}_{0}(\R_{+},\C^{2})$. The domain of the adjoint reads
\[
 \dom(\widetilde{\cD^{*}_{e'}})=H^{1}(\R_{+},\C^{2})
\]
and the trace operators $\Gamma^{e'}_{0,1}:H^{1}(\R_{+},\C^{2})\rightarrow\C$ can be properly given by
\[
\Gamma^{e'}_{0}\begin{pmatrix}\psi^{1}_{e'} \\[.2cm] \psi^{2}_{e'}  \end{pmatrix}= \imath c \psi^{1}_{e'}(0),\qquad \Gamma^{e'}_{1}\begin{pmatrix}\psi^{1}_{e'} \\[.2cm] \psi^{2}_{e'}  \end{pmatrix}= \psi^{2}_{e'}(0).
\]
Again, the gamma field and the Weyl function are provided by \eqref{gammafield} (with respect to the boundary triplet $\{ \gH_{e'},\Gamma^{e'}_{0},\Gamma^{e'}_{1}\}$, with $\gH_{e'}=\C$), while the defect indices are $n_{\pm}(\widetilde{\cD_{e'}})=1$. In addition, one can define again a self-adjoint (by construction) operator as
\[
\cD_{e'}:=\widetilde{\cD^{*}_{e'}},\qquad\dom(\cD_{e'}):=\operatorname{ker}\Gamma^{e}_{0}.
\]

Then, we can describe the Dirac operator introduced in Definition \ref{defi-dirac} using Boundary Triplets. Let $\cD_{0}$ be the operator
\[
\cD_{0}:=\bigoplus_{e\in \mathrm{E}_{s}}\cD_{e}\oplus\bigoplus_{e'\in \mathrm{E}_{h}}\cD_{e'},
\]
defined on $\cH=\bigoplus_{e\in\mathrm{E}_{s}}\cH_{e}\oplus\bigoplus_{e'\in\mathrm{E}_{h}}\cH_{e'}$,
with domain given by the direct sum of the domains of the summands. Consider, also, the operator
\[
\widetilde{\cD}:=\bigoplus_{e\in \mathrm{E}_{s}}\widetilde{\cD_{e}}\oplus\bigoplus_{e'\in \mathrm{E}_{h}}\widetilde{\cD_{e'}},
\]
and its adjoint
\[
\widetilde{\cD^{*}}:=\bigoplus_{e\in \mathrm{E}_{s}}\widetilde{\cD^{*}_{e}}\oplus\bigoplus_{e'\in \mathrm{E}_{h}}\widetilde{\cD^{*}_{e'}}
\]
(with the natural definitions of the domains). Introduce, also, the trace operators
\[
 \Gamma_{0,1}=\bigoplus_{e\in \mathrm{E}_{s}}\Gamma^{e}_{0,1}\oplus\bigoplus_{e'\in \mathrm{E}_{h}}\Gamma^{e'}_{0,1}.
\] 

One can check that $\left\{\gH,\Gamma_{0},\Gamma_{1} \right\}$, with $\gH=\C^{M}$ and $M=2\vert\mathrm{E}_{s}\vert+\vert\mathrm{E}_{h}\vert$,  is a boundary triplet for the operator $\widetilde{\cD^{*}}$ (and hence one can compute gamma-field and Weyl function as before). On the other hand, note that boundary conditions \eqref{eq-kirchtype1}-\eqref{eq-kirchtype2} are ``local'', in the sense that at each vertex they are expressed independently of the conditions on other vertices. As a consequence, they can be written using proper block diagonal matrices $A,B\in \C^{M\times M}$, with $AB^{*}=BA^{*}$, as
\[
A\Gamma_{0}\psi=B\Gamma_{1}\psi
\]
(see also the model case in Section \ref{sec-model}). The sign convention of \eqref{eq-kirchtype2} is incorporated in the definition of the matrix $B$. Therefore the Dirac operator with Kirchoff-type conditions can be defined as 
\[
\cD:=\widetilde{\cD^{*}},\qquad \dom(\cD):=\operatorname{ker}(A\Gamma_{0}-B\Gamma_{1}),
\]
and then, by construction,

\begin{thm}
 The Dirac operator with Kirchhoff-type boundary conditions $\cD$, defined by Definition \ref{defi-dirac} is self-adjoint.
\end{thm}

\begin{remark}
The boundary triplets method provides an alternative way to prove the self-adjointness of $\cD$. More classical approaches \emph{\`{a} la} Von Neumann can be found in \cite{BT-JMP}.
\end{remark}


\subsection{Essential spectrum}

Now, we can focus on the essential spectrum of $\cD$. It can be studied adapting the strategy used for the Schr\"{o}dinger case in \cite{KS-JPA}.

Preliminarily, note that the spectrum of $\cD_{0}$ is given by the union of the spectra of each summand, that is
\[
\sigma(\cD_{0})=\bigcup_{e\in \mathrm{E}_{s}}\sigma(\cD_{e})\cup\bigcup_{e'\in \mathrm{E}_{h}}\sigma(\cD_{e'}).
\]
Precisely, following \cite{CMP-JDE}, each segment $I_{e}$, $e\in \mathrm{E}_{s}$, contributes to the point spectrum of $\cD_{0}$ with eigenvalues given by
\[
\sigma(\cD_{e})=\sigma_{p}(\cD_{e})=\left\{\pm\sqrt{\frac{
2mc^2\pi^2}{ \ell_{e}^2}\,\left(j+\frac12\right)^{2}+m^{2}{c^{4}}} \ , \ j\in\mathbb{N}\right\}, \qquad\forall e\in \mathrm{E}_{s},
\]
and each half-lines has a purely absolutely continuous spectrum
\[
\sigma(\cD_{e})=\sigma_{ac}(\cD_{e})=(-\infty,-mc^{2}]\cup[mc^{2},+\infty),\qquad \forall e\in \mathrm{E}_{h}.
\]

Now, one can check that a Krein-type formula for the resolvent operators holds, namely
\begin{equation}\label{krein}
(\cD-z)^{-1}-(\cD_{0}-z)^{-1}=\gamma(z)\left(B\,M(z)-A\right)^{-1}B\gamma^{*}(\overline{z}), \qquad \forall z\in\rho(\cD)\cap\rho(\cD_{0})
\end{equation}
(with $\gamma(\cdot)$ and $M(\cdot)$ the gamma-field and the Weyl function, respectively, associated with $\cD$ -- see \cite{CMP-JDE}). Hence, the resolvent of the operator $\cD$ is as a perturbation of the resolvent of the operator $\cD_{0}$. Since one can prove that the operator at the right-hand side of \eqref{krein} is of finite rank,
Weyl's Theorem \cite[Thm XIII.14]{RS-IV} gives

\begin{thm}
\label{saspectrum}
The essential spectrum of the operator $\cD$ introduced by Definition \ref{defi-dirac} is given by
\[
\sigma_{ess}(\cD)=\sigma_{ess}(\cD_0)=(-\infty,-mc^2]\cup[mc^2,+\infty).
\]
\end{thm}


\subsection{Absence of eigenvalues in the spectral gap}

A natural question raised by Theorem \ref{saspectrum} concerns the existence of eigenvalues and their location.

Unfortunately, there is no easy and general answer. In principle, eigenvalues can be located {}{in different parts of} the spectrum, {}{depending to} the topology and the metric of the graph. Indeed, they can be embedded in the essential spectrum or at the thresholds {}{(see next Sections for examples)}. However, they cannot be in the spectral gap, as shown by the following computation.

Let $\lambda\in\sigma(\cD)$ be an eigenvalue. As a consequence, there exists $0\neq\psi\in \dom(\cD)$ such that
\[
\cD\psi=\lambda\psi,
\]
or equivalently, such that
\begin{gather}
\label{eigenequation1}-\imath c\frac{d\psi^{2}}{dx}=(\lambda-mc^{2})\psi^{1}, \\[.2cm]
\label{eigenequation2}-\imath c\frac{d\psi^{1}}{dx}=(\lambda+mc^{2})\psi^{2}.
\end{gather}
If $\vert\lambda\vert\neq m$, then we can divide both sides of \eqref{eigenequation2} by $(\lambda+mc^{2})$ and plug the value of $\psi^{2}$ into \eqref{eigenequation1}, obtaining
\be
\label{eigenlaplace}
-c^{2}\frac{d^{2}\psi^{1}}{dx^{2}}=(\lambda^{2}-m^{2}c^{4})\psi^{1}.
\ee
Furthermore, using \eqref{eq-kirchtype1}-\eqref{eq-kirchtype2}, we can prove that
\[
\begin{split}
&\sum_{e\succ v}\frac{d\psi^{1}_{e}}{dx}(\vv)=0, \\[.2cm]
&\psi^{1}_{e_{i}}(\vv)=\psi^{1}_{e_{j}}(\vv),\quad\forall e_{i},e_{j}\succ \vv,
\end{split}
\]
so that $\psi^{1}$ is eigenfunction of the Kirchhoff laplacian on $\cG$. Hence, multiplying \eqref{eigenlaplace} times $\psi^{1}$ and integrating,
\[
\vert\lambda\vert > mc^{2},
\]
namely

\begin{prop}
If $\lambda\in\mathbb{R}$ is an eigenvalue of the operator $\cD$ (defined by Definition \ref{defi-dirac}) then $\vert\lambda\vert\geq mc^2$.
\end{prop}


\subsection{Graphs with eigenvalues at the thresholds}\label{sec-thresholds}

As already remarked, $\cD$ may present eigenvalues at thresholds. This is the content of the following

\begin{prop}
\label{prop-threshold}
 Let $\cG$ be a graph with two \emph{terminal edges} incident at the same vertex $\rm{v}\in\cK$. Then $\lambda=\pm mc^2$ are eigenvalue of the operator $\cD$ (defined by Definition \ref{defi-dirac}).
\end{prop}
 
\begin{remark}
 For simplicity we prove the result for the case depicted in Figure \ref{fig-grafoterminal}, which is the simplest one having the property stated above. The same proof applies to more general graphs provided that they present at least two terminal edges, simply considering spinors which vanish identically everywhere except on the two terminal edges.
 \end{remark}
 
 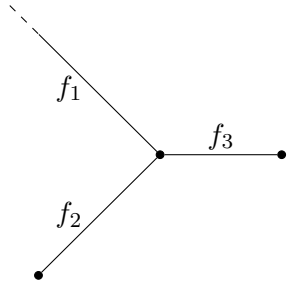
\begin{figure}[h]
 \centering
 \begin{tikzpicture}[xscale= 0.4,yscale=0.4]
 \node at (0,0) [nodo] (00) {};
 \node at (4,0) [nodo] (4) {};
 \node at (-4,4) [infinito] (-44) {};
 \node at (-5,5) [infinito] (-4545) {};
 \node at (-4,-4) [nodo] (-4-4) {};
 \node at (-3,2.2) [infinito] () {$f_1$};
 \node at (-3,-2) [infinito] () {$f_2$};
 \node at (2,.6) [infinito] () {$f_3$};

 \draw [-] (00) -- (4);
 \draw [-] (00) -- (-44);
 \draw [dashed] (-44) -- (-4545);
 \draw [-] (00) -- (-4-4);
 \end{tikzpicture}
 \caption{A 3-star graph with two segments.}
 \label{fig-grafoterminal}
\end{figure}
 
 \begin{proof}[Proof of Proposition \ref{prop-threshold}]
 First, identify the bounded edges of the graph in Figure \ref{fig-grafoterminal} with the compact intervals $I_j=[0,\ell_j], j=2,3$, and the common vertex with $0$.
 
Let $\lambda=mc^2$. Then, equations \eqref{eigenequation1} and \eqref{eigenequation2} read
\begin{gather*}
\frac{d\psi^{2}}{dx}=0\,, \\[.2cm]
\frac{d\psi^{1}}{dx}=2\imath mc\psi^{2}.
\end{gather*}
Now, let $\psi_{f_1}\equiv 0$. Integrating the above equations on $f_2, f_3$ yields
\[
\psi^1_{f_j} (x)= 2\imath mc A_j x+B_j\,,\qquad\mbox{for $x\in [0,\ell_j]$, with $j=2,3$}
\]
and
\[ 
\psi^2_{f_j}(x)\equiv A_{j}\,,\qquad\mbox{for $x\in [0,\ell_j]$, with $j=2,3$}\,
\]
where $A_j ,B_j \in \C$. Therefore, as \eqref{eq-kirchtype1} and \eqref{eq-kirchtype2} have to be satisfied at $\rm{v}\simeq 0$, we find
\[
B_{2}=B_{3}=0,\qquad\mbox{and}\qquad A_{3}=-A_2
\]
and thus $\lambda=mc^2$ is an eigenvalue  of $\cD$.

Let us, now, turn to $\lambda=-mc^2$. In this case the eigenvalue equation becomes
\begin{gather*}
\frac{d\psi^{2}}{dx}=-2\imath mc\psi^{1}, \\[.2cm]
\frac{d\psi^{1}}{dx}=0\,.
\end{gather*}
Setting again $\psi_{f_1}\equiv 0$, we have
\[
\psi^1_{f_j}(x)\equiv E_{j}\,,\qquad\mbox{for $x\in [0,\ell_j]$, with $j=2,3$}\,
\]
and
\[
\psi^2_{f_j} (x)= -2\imath mc E_j x+F_j\,,\qquad\mbox{for $x\in [0,\ell_j]$, with $j=2,3$}
\]
where $E_j, F_j \in \C$, and again by \eqref{eq-kirchtype1} and \eqref{eq-kirchtype2}
\[
E_{2}=E_{3}=0,\qquad\mbox{and}\qquad F_{3}=-F_2.
\]
Then, also $\lambda=-mc^2$ is an eigenvalue of $\cD$ and the proof is completed.
\end{proof}


\subsection{Graphs with embedded eigenvalues}
 
As we mentioned before, it is also possible, properly tuning the topology and the metric of the graph, to give rise to eigenvalues embedded in the essential spectrum.
 
 Here we limit ourselves to show a simple case, which nevertheless displays all the most important features of the phenomenon: the tadpole graph (see, e.g., Figure \ref{fig-tadpole}).
 
 \begin{figure}[ht]
\centering
\includegraphics[width=.6\columnwidth]{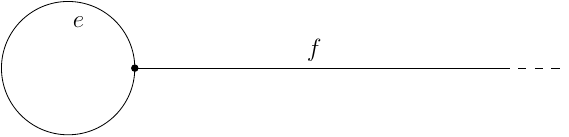}
\caption{Tadpole graph.}
\label{fig-tadpole}
\end{figure}

Arguing as before, first one identifies the circle $e$ of the graph in Figure \ref{fig-tadpole} with the compact interval $I=[0,\ell]$, with $x=0$ and $x=\ell$ representing the same vertex $\vv$.

Then, assume $|\lambda|>mc^2$ in \eqref{eigenequation1} and  \eqref{eigenequation2} and set $\psi_{f}\equiv 0$. In other words, we search for eigenvalues $\lambda\in(-\infty,-mc^2)\cup(mc^2,+\infty) $ whose eigenfunctions are supported on the circle. Namely, we look for spinors $\psi_e\in H^1(I,\C^2)$, such that
\begin{gather*}
-\imath c\frac{d\psi_e^{2}}{dx}=(\lambda-mc^{2})\psi_e^{1}, \qquad\text{in}\quad(0,\ell),\\[.2cm]
-\imath c\frac{d\psi_e^{1}}{dx}=(\lambda+mc^{2})\psi_e^{2}, \qquad\text{in}\quad(0,\ell),
\end{gather*}
and 
\[
\psi_e^1(0)=\psi_e^1(\ell)\qquad , \qquad\psi_e^2(0)=-\psi_e^2(\ell)\,.
\]
This is clearly equivalent to finding $\psi_e^1\in H^2(I)$ such that
\[
 \frac{d^2\psi_e^{1}}{dx^2}=\frac{m^2c^4-\lambda^2}{c^2}\psi_e^1,\qquad\text{in}\quad(0,\ell),
\]
with
\[
\psi_e^1(0)=\psi_e^1(\ell)\qquad\text{and}\qquad\tfrac{d\psi_e^1}{dx}(0)=-\tfrac{d\psi_e^1}{dx}(\ell).
\]
Now, it is easy to check that there are infinitely many values of $\lambda\in(-\infty,-mc^2)\cup(mc^2,+\infty)$ for which the problem admits a solution, i.e.
\[
 \lambda_k:=\mathrm{sgn}(k)\,\sqrt{m^2c^4+\frac{4\pi^2c^2}{\ell^2}k^2},\qquad k\in\Z\setminus\{0\},
\]
and easy computations yield that solutions are of the form
\[
 \begin{array}{l}
  \displaystyle \psi_e^1(x)=A\sin\left(k\pi\left(\frac{2x}{\ell}+1\right)\right)\\[.5cm]
  \displaystyle \psi_e^2(x)=\frac{-\imath2\pi kcA}{\ell(\lambda_k+mc^2)}\cos\left(k\pi\left(\frac{2x}{\ell}+1\right)\right)
 \end{array} 
 \qquad k\in\Z\setminus\{0\}
\]
with $A\in\R$. Hence, we found a sequence of eigenvalues $(\lambda_k)_{k\in\Z\setminus\{0\}}$ embedded in the essential spectrum, unbounded both from below and from above.


\section{A model case: the triple junction}\label{sec-model}

The aim of the present section is to clarify the main ideas explained before by means of an example. 

Consider a 3-star graph with one bounded edge and two half-lines, as depicted in Figure \ref{fig-grafoesempio}. In this case the finite edge is identified with the interval $I=[0,\ell_{e_{3}}]$ and $0$ corresponds to the common vertex of the segment and the half-lines. Here trace operators can be defined as
\[
\Gamma_{0}\psi=\left(\begin{array}{c}\psi_{e_1}^{1}(0) \\[.2cm]\psi_{e_2}^{1}(0) \\[.2cm] \psi_{e_3}^{2}(\ell_{e_{3}}) \\[.2cm] \imath c\psi_{e_3}^{1}(\ell_{e_{3}})\end{array}\right),\qquad \Gamma_{1}\psi=\left(\begin{array}{c}\imath c\psi_{e_1}^{2}(0) \\[.2cm] \imath c \psi_{e_2}^{2}(0) \\[.2cm]\imath c \psi_{e_3}^{2}(0) \\[.2cm] \psi_{e_3}^{1}(\ell_{e_{3}})\end{array}\right),
\] 
so that, again, Kirchoff-type conditions \eqref{eq-kirchtype1}-\eqref{eq-kirchtype2} can be written as 
\[
A\Gamma_{0}\psi=B\Gamma_{1}\psi, 
\]
with $AB^{*}=BA^{*}$ given by
\[
A=\frac{2}{3}\left(\begin{array}{cccc}-2 & 1 & 1 & 0 \\[.2cm] 1 & -2 & 1 & 0 \\[.2cm] 1 & 1 & -2 & 0 \\[.2cm]0 & 0 & 0 & a\end{array}\right),\qquad B=-\imath\frac{2}{3}\left(\begin{array}{cccc}1 & 1 & 1 & 0 \\[.2cm] 1 & 1 & 1 & 0 \\[.2cm] 1 & 1 & 1 & 0 \\[.2cm]0 & 0 & 0 & b\end{array}\right)
\]
(where, properly choosing $a,b\in\mathbb{C}$, one can fix the value of the spinor on the non-connected vertex).

\begin{figure}[h]
 \centering
 \begin{tikzpicture}[xscale= 0.4,yscale=0.4]
 \node at (0,0) [nodo] (00) {};
 \node at (4,0) [nodo] (4) {};
 \node at (-4,4) [infinito] (-44) {};
 \node at (-5,5) [infinito] (-4545) {};
 \node at (-4,-4) [infinito] (-4-4) {};
 \node at (-5,-5) [infinito] (-45-45) {};
 \node at (-3,2.2) [infinito] () {$e_1$};
 \node at (-3,-2) [infinito] () {$e_2$};
 \node at (2,.4) [infinito] () {$e_3$};

 \draw [-] (00) -- (4);
 \draw [-] (00) -- (-44);
 \draw [dashed] (-44) -- (-4545);
 \draw [-] (00) -- (-4-4);
 \draw [dashed] (-4-4) -- (-45-45);
 \end{tikzpicture}
 \caption{A 3-star graph with a finite edge.}
 \label{fig-grafoesempio}
\end{figure}
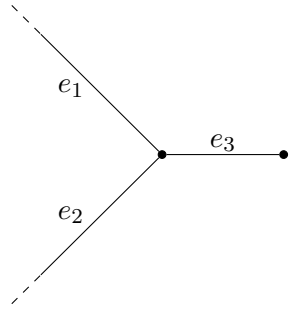

\begin{remark}
 Looking at the example above, and recalling that conditions \eqref{eq-kirchtype1}-\eqref{eq-kirchtype2} are defined independently on each vertex, one can easily see how to iterate the above construction for a more general graph structure, thus obtaining matrices $A,B$ with a block structure, each block corresponding to a vertex. 
\end{remark}

In order to investigate the spectral properties of $\cD$ on the graph depicted by Figure \ref{fig-grafoesempio}, we have to use Weyl's Theorem and compute the singularities of the resolvent (see \eqref{krein}). Hence, it is necessary to compute the gamma-field $\gamma(z)$ and the Weyl function $M(z)$.

The structure of $\gamma(z)$ and $M(z)$ can be recovered using a block composition as in \cite{CMP-JDE, GT-p}. In particular, the eigenvalues of the operator $\cD$ are given by the zeroes of the determinant of $(BM(z)-A)$, so that the computation of $(BM(z)-A)^{-1}B$, is needed.

\begin{remark}
In order to simplify some notations in the following, we choose $m=1/2$ in the definition of the operator \eqref{eq-dirac_formal}, so that the thresholds of the spectrum become $\lambda=\pm c/2$.
\end{remark}

Let us define
\[
k(z):=\frac{1}{c}\sqrt{z^2-(c^2/ 2)^2},\qquad z\in\C\,,
\]
and
\[
k_1(z):=\frac{ck(z)}{z+c^2/2}=\sqrt{\frac{z-c^2/2}{z+c^2/2}}\,,\qquad z\in\C\,.
\]
where  the branch of the multifunction $\sqrt\cdot$ is selected such that $k(x)> 0$ for $x > c^{2}/2$. It this way $k(\cdot)$ is holomorphic in $\mathbb{C}$ with two cuts along the half-lines $(-\infty,-c^{2}/2]$ and $[c^{2}/2,\infty)$.

Then, the Weyl function reads
\[
M(z)=\left(
\begin{array}{cccc}
 i c k_1(z ) & 0 & 0 & 0 \\
 0 & i c k_1(z) & 0 & 0 \\
 0 & 0 &  \frac{c k_1(z ) \sin (\ell_{e_{3}} k(z ))}{\cos (\ell_{e_{3}} k(z ))} & \frac{1}{\cos (\ell_{e_{3}} k(z ))} \\
 0 & 0 & \frac{1}{\cos (\ell_{e_{3}} k(z ))} & \frac{\sin (\ell_{e_{3}} k(z ))}{c k_1(z ){\cos (\ell_{e_{3}} k(z ))}} \\
\end{array}
\right),
\]
Assume also, for the sake of simplicity, that $a=0$, $b=1$ and fix $\ell_{e_{3}}=c=1$. After some calculations, one sees that the zeroes of the determinant of $B M(z)-A$ are given by the zeroes of the following function:

\begin{equation}
\label{det}
f(z)=-\frac{8}{9}  i \left(\sin \left(\frac{1}{2} \sqrt{4 z^2-1}\right)+\sin \left(\frac{3}{2} \sqrt{4 z^2-1}\right)+2 i\right) \sec ^4\left(\frac{1}{2} \sqrt{4 z^2-1}\right).
\end{equation}

\begin{figure}[ht]
\centering
\includegraphics[width=.6\columnwidth]{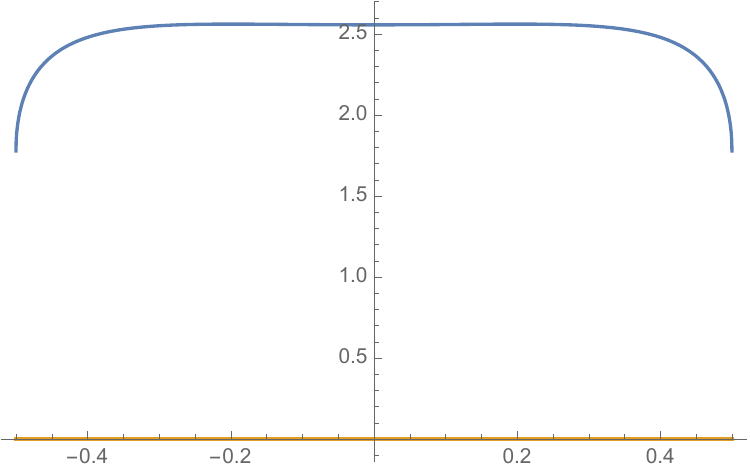}
\caption{The real and imaginary part of $f(z)$ for $z\in[-1/2,1/2]$, in blue and yellow, respectively. }
\end{figure}

Let us consider, initially, $z\in(-1/2, 1/2)$. For such values $\sqrt{4z^2-1}\in \imath\R$, and then (\ref{det}) rewrites
\[
f(z)=\frac{8}{9}   \left(\sinh \left(\frac{1}{2} \sqrt{1-4 z^2}\right)+\sinh \left(\frac{3}{2} \sqrt{1-4 z^2}\right)+2 \right) \textrm{sech} ^4\left(\frac{1}{2} \sqrt{1-4 z^2}\right). 
\]
which is a positive real function. For $z\notin(-1/2,1/2)$, on the contrary, $\sqrt{4 z^2-1}\in\R$ and so 
\[
\Re(f(z))=\frac{16}{9}\sec ^4\left(\frac{1}{2} \sqrt{4 z^2-1}\right)
\]
and 
\[
\Im(f(z))=-\frac{8}{9}\sin \left(2\sqrt{4 z^2-1}\right)\cos \left( \sqrt{4 z^2-1}\right)\,, 
\]
so that the real part of \eqref{det} cannot vanishes, while the imaginary part is periodic and with alternate sign. We can thus conclude that there are eigenvalues in the spectrum. This should be compared with the example presented in Section \ref{sec-thresholds}, where threshold eigenvalues do exist for a small change of the graph structure. 

Moreover, we remark that the above analysis does not exclude the presence of resonances. They may appear, when one imposes Kirchoff-type conditions, due to the eigenvalues of the operator given by the direct sum of the Dirac operators on the edge (without any boundary condition) that are embedded in the continuous part of the spectrum. However, although the explicit analysis of the spectrum for a Dirac graph can be complicated when the topological structure is complex, the spectral characterization and the analysis of resonances appear to be even more challenging and interesting problems. 


\section{The form-domain}\label{sec-quadraticform}

We conclude the paper with some remarks concerning the quadratic form associated with $\cD$. The reason for which this point deserves some particular attention can be easily explained.

In the standard cases ($\R^d$, with $d=1,2,3$) the quadratic form associated with the Dirac operator can be easily defined using the Fourier transform (see, e.g., \cite{ES-CMP}). Unfortunately, in the framework of non compact metric graphs this tool is not available. In addition, also classical duality arguments seems to be prevented as, generally speaking, $H^{-1/2}(\cG)$ is not the topological dual of $H^{1/2}(\cG)$, due to the presence of the compact core of the graph.

Therefore, one has to resort to the \emph{spectral theorem}, where the associated quadratic form $\mathcal{Q}_\cD$ and its domain $\dom(\mathcal{Q}_\cD)$ are defined as
\[
 \dom(\mathcal{Q}_\cD):=\bigg\{\psi\in L^2(\cG,\C^2):\int_{\sigma(\cD)}|\nu|\,d\mu^\cD_\psi(\nu)\bigg\},\qquad \mathcal{Q}_\cD(\psi):=\int_{\sigma(\cD)}\nu\,d\mu^\cD_\psi(\nu),
\]
with $\mu^\cD_\psi$ the spectral measure associated with $\cD$ and $\psi$. However, such a definition is very implicit and thus not useful in concrete cases (as, for instance, in \cite{BCT-p}).


On the other hand, a useful characterization of the form domain can be obtained arguing as follows, using \emph{real interpolation theory} (see, e.g., \cite{AF,A-p}). Define the space
\begin{equation}
\label{interpolateddomain}
Y:=\left[L^{2}(\cG,\C^{2}),\dom(\cD)\right]_{\frac{1}{2}},
\end{equation}
namely, the interpolated space of order $1/2$ between $L^2$ and the domain of the Dirac operator. First, note that
\[
Y\hookrightarrow H^{1/2}(\cG,\C^2):=\bigoplus_{e\in\mE} H^{1/2}(I_{e},\C^{2})=\left[L^{2}(\cG,\C^{2}),H^1(\cG,\C^2)\right]_{\frac{1}{2}},
\]
where $Y$ is endowed with the interpolation norm \eqref{eq:interpolation-norm} and $H^{1/2}(\cG,\C^2)$ with the natural norm. Therefore, by Sobolev embeddings,
 \[
  Y\hookrightarrow L^{p}(\cG,\C^{2}), \qquad\forall p\geq 2,
 \]
and, also, the embedding in $L^{p}(\cK,\C^{2})$ is compact, due to the compactness of $\cK$.

Let us prove that indeed
\begin{equation}
 \label{eq-formeq}
 \dom(\mathcal{Q}_\cD)=Y.
\end{equation}
This characterization turns out to be particularly useful, for instance, in the nonlinear case where one studies the existence of standing waves by variational methods \cite{BCT-p}. In order to prove \eqref{eq-formeq}, we exploit again the spectral theorem, but in a different form (see Theorem \ref{thm-cite} below). In particular, it states that, roughly speaking, every self-adjoint operator on a Hilbert space is isometric to a multiplication operator on a suitable $L^{2}$-space. In this sense self-adjoint operators can be ``diagonalized'' in an abstract way.

\begin{thm}{(\cite[thm. VIII.4]{RS-I})}
\label{thm-cite}
Let $H$ be a self-adjoint operator on a separable Hilbert space $\cH$ with domain $\dom(H)$. There exists a measure space $(M,\mu)$, with $\mu$ a finite measure, a unitary operator
\[
U:\cH\longrightarrow L^{2}\left(M,d\mu\right),
\]
and a real valued function $f$ on $M$, a.e. finite, such that
\begin{enumerate}
\item $\psi\in\dom(H)$ if and only if $f(\cdot)(U\psi)(\cdot)\in L^{2}(M,d\mu),$
\item if $\varphi\in U\left(\dom(H) \right)$, then $\left(UHU^{-1}\varphi \right)(m)=f(m)\varphi(m),\quad \forall m\in M$.
\end{enumerate}
\end{thm}

The above theorem, in other words, states that $H$ is isometric to the multiplication operator by $f$ (still denoted by the same symbol) on the space $L^{2}(M,d\mu)$, whose domain is given by
\[
\dom(f):=\left\{\varphi\in L^{2}(M,d\mu) : f(\cdot)\varphi(\cdot)\in L^{2}(M,d\mu)\right\},
\]
endowed with the norm
\[
\Vert\varphi\Vert^{2}_{1}:=\int_{M}(1+ f(m)^{2})\vert\varphi(m)\vert^{2}d\mu(m)
\]
The form domain of $f$ has an obvious explicit definition, as $f$ is a multiplication operator, that is
\[
 \left\{\varphi\in L^{2}(M,d\mu) : \sqrt{|f(\cdot)|}\,\varphi(\cdot)\in L^{2}(M,d\mu)\right\}
\]
and we will prove in the sequel that it satisfies \eqref{eq-formeq} (we follow the presentation given in \cite{AF,A-p}).

Consider the Hilbert spaces $\cH_{0}:=L^{2}(M,d\mu)$ with the norm $\Vert x\Vert_{0}:=\Vert x\Vert_{L^{2}(d\mu)}$, and $\cH_{1}:=\dom(f)$, so that $\cH_{1}\subset\cH_{0}$. The squared norm $\Vert x\Vert^{2}_{1}$ is a densely defined quadratic form on $\cH_{0}$, represented by
\[
\Vert x\Vert^{2}_{1}=\langle (1+f^{2}(\cdot))x,x\rangle_{0},
\]
where $\langle\cdot,\cdot\rangle_{0}$ is the scalar product of $\cH_{0}$. Define, in addition, the following quadratic version of \emph{Peetre's K-functional}
\[
K(t,x):=\inf\left\{\Vert x_{0}\Vert^{2}_{0}+t\Vert x_{1}\Vert^{2}_{1} : x=x_{0}+x_{1},x_{0}\in \cH_{0},x_{1}\in \cH_{1}\right\}.
\]
By standard arguments (see e.g. \cite{A-p} or \cite[Ch. 7]{AF} and references therein) the intermediate spaces $\cH_{\theta}:=\left[\cH_{0},\cH_{1}\right]_{\theta}\subset\cH_{0}$, $0<\theta<1$, are given by the elements $x\in\cH_{0}$ such that the following quantity is finite:
\begin{equation}\label{eq:interpolation-norm}
 \Vert x\Vert^{2}_{\theta}=\int^{\infty}_{0}\left(t^{-\theta} K(t,x)\right)\frac{dt}{t}<\infty.
\end{equation}

\begin{remark}
{}{In the interpolation procedure we endow $\mathcal{H}_1$ with the $H^1$-norm, as a direct computation shows that it is equivalent to the graph-norm of $\cD$.}
\end{remark}

Now we can prove the following equivalence.
\begin{prop}
For every $\theta\in(0,1)$, one has
\be
\label{eq-thetanorm}
 \Vert x\Vert^{2}_{\theta}=\langle(1+f^{2}(\cdot))^{\theta}x,x \rangle_{0},\qquad\mbox{for $x\in\cH_{\theta}$.}
\ee
\end{prop}

\begin{proof}
Preliminarily, for ease of notation, set $A:=(1+f^2(\cdot))$. The operator $A$ is positive and densely defined on $\cH_{0}$ and also its (positive) square root $A^{1/2}$ has a domain $\cA$ dense in $\cH_{0}$. Now, let us divide the proof in two step.

\emph{Step 1. There holds}
\be\label{eq-KA} K(t,x)=\left\langle\frac{tA}{1+tA}x,x \right\rangle_{0}, \qquad t>0,\quad x\in\cH_{0}.
\ee
First, observe that the  bounded operator in \eqref{eq-KA} is defined via the functional calculus for $A$. Then, take $x\in\cA$ and fix $t>0$. By a standard convexity argument one gets the existence of a unique decomposition
\[
x=x_{0,t}+x_{1,t},
\]
such that
\be
\label{eq-dec1}
K(t,x)=\Vert x_{0,t}\Vert^2_0 +t\Vert x_{1,t}\Vert^2_1 
\ee
(note also that $x_{j,t}\in\cA $, $j=0,1$). Then, for all $y\in\cA$, using the minimality requirement in the definition of $K$, one gets
\[
\frac{d}{ds}\left(\Vert x_{0,t}+sy\Vert^2_0+t\Vert x_{1,t}-sy\Vert^2_1 \right)_{\vert s=0}=0,
\]
and then, recalling that $\Vert x\Vert^2_1=\Vert A^{1/2}x\Vert^2_0$, we have
\[
\langle A^{-1/2}x_{0,t}-tA^{1/2}x_{1,t},A^{1/2}y \rangle_0=0.
\]
Since the above inequality must be true for all $y$ in the dense subset $\cA\in\cH_{0}$ we conclude that
$$A^{-1/2}x_{0,t}=tA^{1/2}x_{1,t}, $$
so that we obtain
\be\label{eq-dec2}
x_{0,t}=\frac{tA}{1+tA}x,\qquad x_{1,t}=\frac{1}{1+tA}x 
\ee
Combining \eqref{eq-dec1} and \eqref{eq-dec2} we get the claim. 
 
\emph{Step 2. Proof of \eqref{eq-thetanorm}.}
By Step 1 and exploiting the functional calculus for $A$, we get
\begin{multline}
\label{eq-rewrite}\Vert x\Vert^2_{\theta}=\int^{\infty}_{0}t^{-\theta}K(t,x)\frac{dt}{t}=\int^{\infty}_{0}t^{-\theta}\left\langle\frac{A}{1+tA}x,x\right\rangle_{0}dt\\[.2cm]
=\left\langle A\left(\int^{\infty}_{0}\frac{dt}{t^{\theta}(1+tA)}\right)x,x\right\rangle_{0}.
\end{multline}
Consider, then, the differentiable function
\[
f(a):=\int^{\infty}_{0}\frac{dt}{t^{\theta}(1+ta)},\qquad a>0.
\]
Integrating by parts, one easily gets
\[
\int^{\infty}_{0}\frac{dt}{t^{\theta}(1+ta)}=\frac{a}{1-\theta}\int^{\infty}_{0}\frac{tdt}{t^{\theta}(1+ta)^2}=-\frac{a}{1-\theta}f'(a).
\]
Then $f$ fulfills
\[
f'(a)=\frac{(\theta-1)}{a}f(a),\qquad a>0,
\]
and integrating 
\[
f(a)=a^{\theta-1}.
\]
Note that we have set the integration constant equal to zero in order to get the correct formula as $\theta\rightarrow 1^-$. Combining the above observations, one sees that \eqref{eq-rewrite} reads
$$\Vert x\Vert^{2}_{\theta}=\langle A^{\theta}x,x\rangle_{0}, $$ thus proving the claim.
\end{proof}

\begin{remark}
 Observe that, in the proof above, the bounded operators 
 $$
 \frac{tA}{1+tA},\qquad \frac{1}{1+tA}\,,\qquad t\geq 0\,,
 $$
 are defined using the functional calculus for the self-adjoint operator $A$.
 \end{remark}

Finally, in view of the previous proposition, if one sets for $\theta=\frac{1}{2}$, then one recovers the form domain of the operator $f$ and, hence, setting $H=\cD$ and $\cH=L^{2}(\cG,\C^{2})$, one has that \eqref{interpolateddomain} is exactly the form domain of $\cD$, with $Y=U^{-1}\cH_{\frac{1}{2}}$. Consequently, \eqref{eq-formeq} is satisfied and, summing up, we have shown the following

\begin{thm}
The form domain of $\cD$ (defined by Definition \ref{defi-dirac}) satisfies
\[
\dom(\mathcal{Q}_\cD)=\left[L^{2}(\cG,\C^{2}),\dom(\cD)\right]_{\frac{1}{2}},
\]
namely, is equal to the interpolated space of order $1/2$ between $L^2$ and the operator domain.
\end{thm}


\end{document}